\documentclass[12pt,a4paper]{amsart}
\usepackage[T1]{fontenc}
\usepackage{mathrsfs}
\usepackage{amsmath,amsthm,amsfonts,amssymb,amscd,graphicx}
\usepackage[colorlinks,
            linkcolor=blue,       
            anchorcolor=blue,  
            citecolor=blue,        
            ]{hyperref}
\usepackage{graphicx}

\usepackage{quiver}                     
\usepackage{subcaption}                 
\usepackage{booktabs}                   

\tikzcdset{scale cd/.style={every label/.append style={scale=#1},
    cells={nodes={scale=#1}}}}          

\usepackage[OT2,T1]{fontenc}
\DeclareSymbolFont{cyrletters}{OT2}{wncyr}{m}{n}
\DeclareMathSymbol{\Be}{\mathalpha}{cyrletters}{"42}
\newtheorem{thm}{Theorem}[section]  
\newtheorem{pro}[thm]{Proposition}  
\newtheorem{cor}[thm]{Corollary}    
\newtheorem{lem}[thm]{Lemma}        
\newtheorem{ex}[thm]{Example}       
\newtheorem{rk}[thm]{Remark} 
\newtheorem{definition}[thm]{Definition}

\title[The finitude and structure of $G_S(K,p)$]{On the Finiteness and Structure of Galois Groups of Tamely ramified pro-p Extensions of Imaginary Quadratic Fields}
\author{Qi Liu}
\author{Zugan Xing}
\date{2024.Oct}

\begin{document}


\address{(Qi Liu) Department of Mathematics, Capital Normal University, Beijing,
China}
\email{QiLiu67@aliyun.com}
\address{(Zugan Xing) Department of Mathematics, Capital Normal University, Beijing,
China}
\email{xingzugan@aliyun.com}
\setlength{\belowcaptionskip}{-5pt}
\setlength{\abovecaptionskip}{-5pt}

\begin{abstract}
For a prime $p$, we study the Galois groups of maximal pro-$p$ extensions of imaginary quadratic fields unramified outside a finite set $S$, where $S$ consists of one or two finite places not lying above $p$. When $p$ is odd, we provide explicit presentations of these Galois groups under certain conditions. As an application, we determine the defining polynomial of the maximal pro-$3$ extension of $\mathbb{Q}(i)$ unramified outside two specific finite places. 
\end{abstract}

\maketitle



\section{Introduction}
\subsection{Background and Previous Work}
Let \( p \) be a rational prime number, \( K \) a number field, and \( S \) a finite set of places of \( K \). We define \( K_S \) to be the maximal pro-\( p \) extension of \( K \)  unramified outside the set \( S \).

In particular, when \( S = \varnothing \), the field \( K_S \) corresponds to the classical \( p \)-class field tower of \( K \), a well-studied object in algebraic number theory. The structure and finiteness of the Galois group \( G_S(K,p) \) of the extension \( K_S / K \) form a long-standing problem in the field. If \( S \) contains places of \( K \) above \( p \), the abelianization \( G_S(K,p)^{\text{ab}} \) can be infinite. For example, when \( K = \mathbb{Q} \) and \( S = \{p\} \), \( G_S(K,p)^{\text{ab}} \) is the cyclotomic \( \mathbb{Z}_p \)-extension of \( \mathbb{Q} \). On the other hand, if \( S \) consists only of places that do not lie over \( p \), then \( G_S(K,p)^{\text{ab}} \) is always finite, a result known as the "tame case" (see \cite{hajir2003tame}). This case, however, has not been fully understood. We will focus on the tame case in this article.

The classical approach to determining the infiniteness of \( G_S(K,p) \) involves the theorem of Golod and Shafarevich \cite{golod1964class}, which states that for a pro-\( p \) group \( G \), if the generator and relation ranks \( d(G) \) and \( r(G) \) satisfy \( \frac{d(G)^2}{4} \geq r(G) \), then \( G \) must be infinite. In the tame case, Farshid Hajir and Christian Maire \cite{hajir2001tamely} extended this method and showed that \( G_S(K,p) \) is infinite if the \( p \)-rank of the ray class group \( \text{Cl}_K(\mathfrak{m}) \) (modulo \( \mathfrak{m} = \prod_{\mathfrak{q} \in S} \mathfrak{q} \)) satisfies the inequality
\[
\text{rk}_p\left( \text{Cl}_K(\mathfrak{m}) \right) \geq 2 + 2 \sqrt{r_K + \theta_{K,S} + 1},
\]
where \( r_K \) is the number of infinite places of \( K \), and \( \theta_{K,S} \) is 1 if \( S \) is empty and \( K \) contains the primitive \( p \)-th roots of unity, and 0 otherwise.

The primary method for further investigations into the finiteness and structure of \( G_S(K,p) \) in the tame case is the lower \( p \)-central series of \( G_S(K,p) \). Specifically, if \( G_S(K,p)^{(i,p)} \) stabilizes at some level \( i \), then \( G_S(K,p) \) is finite. This approach has led to significant results, such as the work of Helmut Koch \cite{koch1965} (also see \cite{koch2002galois}, Chapter 11.5), who identified a family of sets \( S \) formed by two distinct prime numbers, neither dividing the odd prime \( p \), for which \( G_S(\mathbb{Q},p) \) is a non-Abelian group of order \( p^3 \) and exponent \( p^2 \). 

For \( p = 2 \), Boston and Perry \cite{boston2000maximal} showed that if \( S = \{p, q\} \) with \( p, q \equiv 3 \pmod{4} \), then \( G_S(\mathbb{Q},2) \) is a semidihedral group, and for \( p \equiv 3 \pmod{4} \) and \( q \equiv 1 \pmod{4} \), \( G_S(\mathbb{Q},2) \) is a modular group. In subsequent work \cite{boston2002explicit}, Boston and Leedham-Green developed a computational algorithm for determining \( G_S(\mathbb{Q},2) \). Using this approach, Mizusawa \cite{yasushi2016certain} proved that when \( S \) consists of exactly three odd rational primes, \( G_S(\mathbb{Q},2) \) is a finite 2-group of order \( 2^9 \).

For more general number fields, recent work by Yoonjin Lee and Donghyeok Lim \cite{lee2024finitude} has introduced a new method for determining the finiteness of \( G_S(K,p) \) by checking the powerfulness of the group. They examine the finiteness of \( G_S(K,p) \) for fields with a non-trivial cyclic \( p \)-class group, where \( S \) contains exactly one non-\( p \)-adic prime ideal \( \mathfrak{q} \).

\subsection{Main Results and Organization}

In this paper, we continue the work of Boston and Lim by focusing on imaginary quadratic fields. 

Let $K$ be an imaginary quadratic field that does not contain the primitive $p$-th roots of unity. In other words, we consider $p$ to be an odd prime number, and when $p=3$, $K$ is not equal to $\mathbb{Q}(\sqrt{-3})$. We have the following main theorems:

\begin{thm}\label{s=1}
If $p$-class group of \(K\) is isomorphic to $\mathbb{Z}/p\mathbb{Z}$, then:

1. There exist infinitely many non-$p$-adic prime ideals $\mathfrak{q}$ of $K$ that are inert in the $p$-class field $H_p(K)$ of $K$, with the $p$-rank of ray class group $Cl_K(\mathfrak{q})$ greater than 1. 

2. For any such prime ideal $\mathfrak{q}$, we have  
$$G_S(K,p) = \left\langle a,b \mid a^{p^{n-1}} = 1, \, b^p = 1, \, b^{-1}ab = a^{1+p^{n-2}} \right\rangle,$$
where $S = \{\mathfrak{q}\}$ and $n-1 \geq 2$ equals \(\operatorname{ord}_p\left(\left|Cl_{H_p(K)}(\mathfrak{q}')\right|\right)\), with \(\mathfrak{q}'\) being the unique prime ideal of $H_p(K)$ lying above $\mathfrak{q}$.
\end{thm}

Moreover, if the prime ideal $\mathfrak{q}$ satisfies the above conditions and additionally $p \Vert N(\mathfrak{q}) - 1$, then the order of $G_S(K,p)$ must be $p^3$. We will prove this in Corollary \ref{p^3}. When $\#S = 2$, we can also determine the structure of $G_S(K,p)$ under certain conditions.

\begin{thm}\label{thm 1.1}
 Let \( \mathfrak{q}_1 \) be a prime ideal of \( K \) not lying above \( p \) such that \( p \Vert N(\mathfrak{q}_1) - 1 \). Define \( M(\mathfrak{q}_1, p) \) to be the unique degree \( p \) extension of \( K \) contained in the ray class field modulo \( \mathfrak{q}_1 \). For any prime ideal $\mathfrak{q}_2$ of $K$, not lying above $p$, that satisfies $p \mid N(\mathfrak{q}_2) - 1$ and is inert in $M(\mathfrak{q}_1, p)$, the Galois group $G_S(K,p)$ has the following presentation:
\[
G_S(K,p) = \left\langle a, b \mid a^{p^{n-1}} = 1, \, b^p = 1, \, b^{-1}ab = a^{1 + p^{n-2}}\right \rangle,
\]
where $S = \{\mathfrak{q}_1, \mathfrak{q}_2\}$ and $n - 1 \ge 2$ equals the $\operatorname{ord}_p\left(\left|Cl_{M(\mathfrak{q}_1,p)}\left(\mathcal{Q}_1\mathcal{Q}_2\right)\right|\right)$, with $\mathcal{Q}_i$ denoting the prime ideal of $M(\mathfrak{q}_1, p)$ lying above $\mathfrak{q}_i$ for $i = 1, 2$.
\end{thm}

The paper is organized as follows. Section 2 presents essential facts about ray class groups of number fields, which are crucial for proving the finiteness of \(G_S(K,p)\) and calculating its order. In Section 3, we discuss the finiteness of \(G_S(K,p)\) in three different cases. Section 4 contains the proofs of our main structural results (Theorems \ref{s=1} and Theorem \ref{thm 1.1}). 

Finally, in Section 5, we provide explicit computations for the maximal pro-3 extension of \(\mathbb{Q}(i)\) unramified outside a specific set \(S\) of two finite places, including its defining polynomials and intermediate fields.

\section{The ray class group}

Let $K$ be a number field and $\mathfrak{m}$ a modulus of $K$. We denote the ray class field of $K$ modulo $\mathfrak{m}$ by $K(\mathfrak{m})$ and the ray class group by $Cl_K(\mathfrak{m})$. In this section, we assume that $\mathfrak{m}$ is a product of finite places of $K$.

\begin{lem}[\cite{cohen2012advanced}, Propostion 3.2.3.]\label{2.1}
Let $\mathcal{O}_{K}$ be the ring of integers of $K$, $E_{K}$ the unit group of $\mathcal{O}_{K}$ and $E_{K}(\mathfrak{m})$ is the group of units congruent to 1 modulo $\mathfrak{m}$. We have the following exact sequence
$$ 1\to  E_{K} /E_{K}(\mathfrak{m})\stackrel{\rho}{\to} (\mathcal{O} _{K} /\mathfrak{m})^{\times}\to Cl_K(\mathfrak{m})\to Cl_K(1)\to1.$$
\end{lem}

If we assume $\mathfrak{m}=\mathfrak{q}_1 ^{n_{1}}\mathfrak{q}_2^{n_{2}}\cdots \mathfrak{q}_{k}^{n_k}$ where \(\mathfrak{q}_i\) is the prime ideal of \(K\), then 
\begin{equation}
\left( \mathcal{O}_{K}/\mathfrak{m} \right)^{\times} \cong \prod _{\mathfrak{q}|\mathfrak{m}}\left(\mathcal{O} _{K}/\mathfrak{q} ^{v _{\mathfrak{p}}(\mathfrak{m})}\right)^{\times}    
\end{equation}
Furthermore, we have 
\begin{equation}
    |Cl_{K}(\mathfrak{m})|= \frac{ |Cl_{K}(1)| \cdot \prod _{\mathfrak{q}|\mathfrak{m}}\left|\left( \mathcal{O}_{K}/\mathfrak{q}^{v _{\mathfrak{q}}(\mathfrak{m})} \right)^{\times}\right| }{|E _{K}:E _{K}(\mathfrak{m})|}.
\end{equation} 

Hence, if the class number of $K$ is known, the key to obtaining information about $K(\mathfrak{m})$ is to calculate the residue ring $\left( \mathcal{O}_{K}/\mathfrak{q} \right)^{\times}$ and the index of $E_{K}(\mathfrak{m})$ in $E_{K}$. 

Regarding the structure of the residue ring of powers of a prime ideal, Michele Elia, J. Carmelo Interlando, and Ryan Rosenbaum established the following theorems. Set \(\mathfrak{q}\) be a prime ideal of \(K\) lying above the rational odd prime \(q\) with ramification index \(e\) and inertia degree \(f\).

\begin{thm}[\cite{elia2010structure}, Theorems 5 and 6]\label{thm2.3}
If \(e = f = 1\), then the residue ring \(\mathcal{O}_K / \mathfrak{q}^n\) is isomorphic to \(\mathbb{Z} / q^n \mathbb{Z}\). If \(e = 1\) and \(f > 1\), \((\mathcal{O}_K / \mathfrak{q}^n)^\times\) is an abelian group of the form \(G_1(\mathfrak{q}^n) \times G_2(\mathfrak{q}^n)\), where \(G_1(\mathfrak{q}^n)\) is a cyclic group of order \(q^{f}-1\) and \(G_2(\mathfrak{q}^n)\) has order \(q^{f(n-1)}\).
\end{thm}

\begin{thm}[\cite{elia2011structure}, Theorem 2, Theorem 3, and Theorem 4]\label{thm2.4}
If \(e>1\) and \(q\) is an prime number, then \((\mathcal{O}_K / \mathfrak{q}^n)^{\times}\) is isomorphic to the direct product of two abelian groups \(G_1(\mathfrak{q}^n)\) and \(G_2(\mathfrak{q}^n)\), where \(G_1(\mathfrak{q}^n)\) is a cyclic group of order \(q^f - 1\) and \(G_2(\mathfrak{q}^n)\) is an abelian \(q\)-group with order $q^{f(n-1)}$.
\end{thm}

Let $\ell$ be a rational prime. If $G$ is an abelian group, the $\ell$-primary part of $G$ is denoted by $G(\ell)$, and the non-$\ell$-primary components are denoted by $G(\bar{\ell})$. Based on the above classical theorems, we have the following lemmas.

\begin{lem}\label{lem2.4}
Let $p$ be a rational prime number, and let $\mathfrak{q}_i$, for $i=1,\ldots,k$, be finite prime ideals of $K$ that are coprime to $p\mathcal{O}_{K}$ and \(2\mathcal{O}_K\). Let $\mathfrak{m} = \mathfrak{q}_1^{n_1} \cdots \mathfrak{q}_k^{n_k}$ for some $n_i \in \mathbb{N}^\times$. If $n \in \mathbb{N}^\times$ such that $p^n \Vert \left|Cl_K(\mathfrak{q}_1 \cdots \mathfrak{q}_k)\right|$, then $p^n \Vert \left|Cl_K(\mathfrak{m})\right|$.
\end{lem}

\begin{proof}
Let $(u_{i})_{i}$ be the generators of $E_{K}$ as a $\mathbb{Z}$-module. The index $[E_{K}:E_{K}(\mathfrak{m})]$ is the product of the orders of $\rho(u_{i})$, where $\rho(u_{i})$ is the image of $u_{i}$ in $(\mathcal{O}_{K}/\mathfrak{m})^{\times}$.

Fix a generator of $E_{K}$, denoted by $u$. Let $\rho(u)=g_{1}(\mathfrak{m})\times g_{2}(\mathfrak{m})$, where $g_{1}(\mathfrak{m})\in (\mathcal{O}_{K}/\mathfrak{m})^{\times}(p)$ and $g_{2}(\mathfrak{m})\in (\mathcal{O}_{K}/\mathfrak{m})^{\times}(\bar{p})$. The natural surjective map $(\mathcal{O}_{K}/\mathfrak{q}_{i}^{n_i})^{\times} \to (\mathcal{O}_{K}/\mathfrak{q}_{i})^{\times}$ induces an isomorphism $(\mathcal{O}_{K}/\mathfrak{q}_{i}^{n_i})^{\times}(p) \to (\mathcal{O}_{K}/\mathfrak{q}_{i})^{\times}(p)$ by Theorem \ref{thm2.3} and Theorem \ref{thm2.4}. Moreover, we have the following commutative diagram.

\begin{figure}[ht]

\begin{tikzcd}[scale cd=0.6]
	{(\mathcal{O}_K/\mathfrak{m})^\times(p)} && {(\mathcal{O}_K/\mathfrak{q}_1^{n_1})^\times(p)\times \cdots \times (\mathcal{O}_K/\mathfrak{q}_k^{n_k})^\times(p)} \\
	\\
	{(\mathcal{O}_K/\mathfrak{q}_1 \cdots \mathfrak{q}_k)^\times(p)} && {(\mathcal{O}_K/\mathfrak{q}_1)^\times(p)\times \cdots \times (\mathcal{O}_K/\mathfrak{q}_k)^\times(p)} \\
	{x\bmod \mathfrak{m}} && {(x\bmod \mathfrak{q}_1^{n_1},...,x\bmod \mathfrak{q}_k^{n_k})} \\
	\\
	{x\bmod \mathfrak{q}_1\cdots\mathfrak{q}_k} && {(x\bmod \mathfrak{q}_1,...,x\bmod \mathfrak{q}_k)}
	\arrow["\cong", from=1-1, to=1-3]
	\arrow[two heads, from=1-1, to=3-1]
	\arrow["\cong", from=1-3, to=3-3]
	\arrow["\cong", from=3-1, to=3-3]
	\arrow[maps to, from=4-1, to=4-3]
	\arrow[maps to, from=4-1, to=6-1]
	\arrow[maps to, from=4-3, to=6-3]
	\arrow[maps to, from=6-1, to=6-3]
\end{tikzcd}
\vspace{10pt}
\captionof{figure}{The commutative diagram for Lemma \ref{lem2.4}}
\end{figure}

Thus, the order of $g_{1}(\mathfrak{m})$ equals the order of $g_{1}(\mathfrak{q}_1 \cdots \mathfrak{q}_k)$, the image of $u$ in the $p$-primary component of $(\mathcal{O}_K / \mathfrak{q}_1 \cdots \mathfrak{q}_k)^{\times}$. Consequently, if $p^{n'} \Vert [E_{K}:E_{K}(\mathfrak{m})]$, then $p^{n'} \Vert [E_{K}:E_{K}(\mathfrak{q}_1\cdots \mathfrak{q}_k)]$. Moreover, $p^{n''} \Vert \left|(\mathcal{O}_{K}/\mathfrak{m})^{\times}\right|$ if and only if $p^{n''} \Vert \left|(\mathcal{O}_{K}/(\mathfrak{q}_1 \cdots \mathfrak{q}_k))^{\times}\right|$. The results follow from equation (2).
\end{proof}

\begin{lem}\label{lem 2.5}
Let $q$ and $p$ be two distinct odd prime numbers, and let $K$ be an imaginary quadratic field (excluding $\mathbb{Q}(\zeta_3)$ when $p=3$). Let $\mathfrak{q}$ be a prime ideal of $K$ lying above $q$.If $p$ divides $N(\mathfrak{q})-1$ and $K$ has trivial $p$-class group, then for any $n$, $Cl_{K}(\mathfrak{q}^{n})$ contains a  unique subgroup isomorphic to $\mathbb{Z} /p \mathbb{Z}$.
\end{lem}

\begin{proof}
we have following exact sequence
$$ \mathcal{O} _{K}^{\times}(p)\to (\mathcal{O}_{K} /\mathfrak{q}^n)^{\times}(p)\to Cl_K(\mathfrak{q}^n)(p)\to 1.$$
     This is immediate, as the $p$-th unity of roots in $K$ equals 1, and $G_{1}(\mathfrak{q} ^{n})$ is a cyclic group of order $N(\mathfrak{q})-1$. 
\end{proof}

According to class field theory, we have the following corollary, which will be commonly used in the next section.

\begin{cor}\label{2.6}
If the above conditions on $K$, $p$, and $q$ are satisfied, then $K$ has a unique $\mathbb{Z}/p\mathbb{Z}$-extension contained within $K(\mathfrak{q}^n)$ for every $n$. This extension is denoted by $M(\mathfrak{q},p)$. 
\end{cor}

\begin{rk}\label{rk 2.7}
    When \(K\) is an imaginary quadratic field with trivial \(2\)-class group, \(p=2\), and \(q\) is an odd rational prime number, if \(4\mid N(\mathfrak{q})-1\), then there exists a number field \(M(\mathfrak{q},2)\), which is the unique \(\mathbb{Z}/2\mathbb{Z}\)-extension of \(K\) contained in the ray class field of \(K(\mathfrak{q}^n)\) for any \(n\ge 1\).
\end{rk}

\section{\texorpdfstring{The finitude and powerfulness of \(G_S(K,p)\)}{The finitude and powerfulness of G\_S(K)}}

In this section, we use Corollary \ref{2.6} to establish the finiteness of \( G_{S}(K,p) \) and the powerfulness of \( G_{S}(K,p) \) when \( p \) is an odd prime. 

More specifically, we first consider the case where \( p \) is an odd prime. Later, we discuss two cases when \( p = 2 \): when \( S \) contains two finite primes and when \( S \) contains exactly one finite prime. We first introduce the following essential concepts of pro-$p$ groups.

\begin{definition}
Let \( G \) be a pro-\( p \) group. We say that \( G \) is powerful if \( G/\overline{G^p} \) is an abelian group when \( p \) is odd, or if \( G/\overline{G^4} \) is an abelian group when \( p = 2 \).
\end{definition}

\begin{ex}
A pro-\(p\) group \(G\) is a Demushkin group if the generator rank \(d(G) < \infty\), the relation rank \(r(G) = 1\), and the cup product 
$$ 
H^1\left(G, \mathbb{F}_p\right) \times H^1\left(G, \mathbb{F}_p\right) \rightarrow H^2\left(G, \mathbb{F}_p\right) 
$$ is a non-degenerate bilinear form. In particular, if \(d(G) = 2\) and \(p\) is odd, then the Demushkin pro-\(p\) groups \(G\) is powerful.
\end{ex}
\begin{proof}
    This can be established by \cite[Theorem 3.9.11]{neukirch2013cohomology} and \cite[Theorem 3.9.19]{neukirch2013cohomology}.
\end{proof}

From the definitions of powerful and generator rank, the following lemma is immediate.

\begin{lem}\label{3.3}
   Let \(G\) be a pro-\(p\) group and \(H\) a closed normal subgroup of \(G\). Then, we have 
\(d(G/H) \le d(G)\) and \(H\) is a powerful pro-\(p\) group if \(G\) is.
\end{lem}

Depending on the uniform version of the Tame Fontaine-Mazur conjecture, Lee and Lim in \cite{lee2024finitude} proved the following result, which can be used to determine the finitude of $G_S(K,p)$.

\begin{pro} \cite{lee2024finitude} \label{3.4}
Let $F$ be a number field and $S$ a finite set of non-p-adic places of $F$. Let $\mathcal{F}$ be a pro-p extension of $F$ which is unramified outside $S$. If the Galois group $H=\operatorname{Gal}(\mathcal{F} / F)$ is powerful with $d(H) \leq 2$, then $H$ is finite.
\end{pro}

\begin{proof}
    This result is established in \cite{lee2024finitude}, Proposition 2.9.
\end{proof}

\begin{lem}\label{lem 3.5}
Let \(K\) be a number field, \(p\) an odd rational prime number and the set \(S\) of non-\(p\)-adic places of \(K\). Let \(M\) be the fixed subfield of \(K_S\) fixed by the Frattini subgroup of \(G_S(K,p)\). If there exists a finite place \(\mathfrak{q} \in S\) such that \(\mathfrak{q}\) does not split in $M$, then \(G_S(K,p)\) is finite and powerful.   
\end{lem}

\begin{proof}
Since \(\mathfrak{q}\) is not split in \(M\), the decomposition subgroup \(D_\mathfrak{q}\) of \(\operatorname{Gal}(M/K)\) at \(\mathfrak{q}\) equals the entire Galois group. By Burnside's basis theorem, for a place $\mathfrak{Q}$ of $K_{S}$ above $\mathfrak{q}$, the decomposition subgroup of \(G_S(K,p)\) at $\mathfrak{Q}$ equal $G_{S}(K,p)$. 


\begin{figure}[ht]
\[\begin{tikzcd}
	{K_S} \\
	\\
	M \\
	K
	\arrow["{D_{\mathfrak{Q}}}"', curve={height=18pt}, dashed, no head, from=1-1, to=4-1]
	\arrow[no head, from=3-1, to=1-1]
	\arrow[no head, from=4-1, to=3-1]
	\arrow["{D_\mathfrak{q}}"', curve={height=12pt}, dashed, no head, from=4-1, to=3-1]
\end{tikzcd}\]
\vspace{10pt}
\captionof{figure}{The field diagram for Lemma \ref{lem 3.5}}
\end{figure}
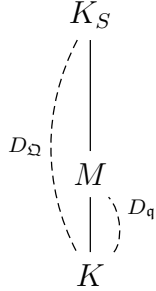

Furthermore, $G_{S}(K,p)$ is isomorphic to a quotient of the Galois group $\operatorname{Gal}(\overline{K_{\mathfrak{q}}}/K_{\mathfrak{q}})$, where $\overline{K_{\mathfrak{q}}}$ is the maximal pro-$p$ extension of the completion $K_{\mathfrak{q}}$ of $K$ at $\mathfrak{q}$. According to [\cite{koch2002galois}, Chapter 10], $\operatorname{Gal}(\overline{K_{\mathfrak{q}}}/K_{\mathfrak{q}})$ is a pro-$p$ Demushkin group of generator rank 2. Since the Demushkin groups of generator rank 2 are powerful, $G_S(K,p)$ is powerful and generator rank less or equal 2 by Lemma \ref{3.3}. The results follows from the proposition \ref{3.4}.    
\end{proof}
 
By the above lemma, for finding some cases of finite \(K_S\), a key point is to find a “ramified” prime ideal whose decomposition subgroup covers the full Galois group $\operatorname{Gal}(M/K)$. To do this, we need a clearer understanding of the $p$-divisibility of the ray class number of number fields and generator ranks. Regarding the \(p\)-divisibility of the ray class number, we have outlined several properties in the preceding section.

Regarding the generator rank, there has an important characterization in \cite{koch2002galois}, which can also be found in Theorem 10.7.10 of \cite{neukirch2013cohomology}. We define:
\[
\delta(K) = \begin{cases} 
1 & \text{if } \zeta_p \subseteq K \\
0 & \text{if } \zeta_p \nsubseteq K,
\end{cases} 
\quad \text{and} \quad 
\delta(K_{\mathfrak{p}}) = \begin{cases} 
1 & \text{if } \zeta_p \subseteq K_{\mathfrak{q}}, \\
0 & \text{if } \zeta_p \nsubseteq K_{\mathfrak{q}},
\end{cases}
\]
where \(\mathfrak{q}\) is a prime ideal of \(K\) and \(K_{\mathfrak{q}}\) is the completion of \(K\) at \(\mathfrak{q}\). Let
\[
V_S = \left\{ \alpha \in K^{\times} \mid \alpha \mathcal{O}_K = \mathfrak{a}^p \text{ and } \alpha \in K_{\mathfrak{p}}^p \text{ for all } \mathfrak{p} \in S \right\}.
\]
And the group \(\Be_S\) is the Pontryagin dual of \(V_S/K^{\times p}\).

\begin{pro}[\cite{koch2002galois}, Theorem 11.8]\label{2.5}
    Let \(K\) be a number field, \(S\) a finite set of non-\(p\)-adic places of \(K\), and \(r\) the number of infinite places of \(K\). Then
\[ d(G_S(K,p)) = 1 - r - \delta(K) + \sum_{\mathfrak{q} \in S} \delta(K_{\mathfrak{q}}) + \dim_{\mathbb{F}_p} \Be_S. \]
\end{pro}

\subsection{The case of odd \texorpdfstring{$p$}{p}}

\begin{thm}\label{finite of two q}
    Let $K$ be an imaginary quadratic field with a trivial $p$-class group, where $p$ is an odd rational prime satisfying $p \nmid |\mu _{K}|$. For any non $p$-adic prime ideal $\mathfrak{q}_1$ of $\mathcal{O}_K$, there exist infinite many prime ideals $\mathfrak{q}_2$ of $\mathcal{O}_{K}$ such that $G_S(K)$ is finite powerful $p$-group, where $S = \{\mathfrak{q}_1, \mathfrak{q}_2\}$.  
\end{thm}

\begin{proof}
    
Without loss of generality, we assume that $N(\mathfrak{q} _{1})\equiv 1 \bmod p$. By Corollary \ref{2.6}, the ray class field modulo $\mathfrak{q}_1$ of $K$ contains exactly one $\mathbb{Z}/p\mathbb{Z}$ subextension, denoted \(M(\mathfrak{q}_1,p)\). We consider the field extension diagrams as follows:   

\begin{figure}[ht]
\begin{tikzcd}[scale cd=0.7]
	&{M(\mathfrak{q}_1,p)(\zeta_p)} \\
	{K(\zeta_p)} & & {M(\mathfrak{q}_1},p) \\
	& K
	\arrow[no head, from=2-1, to=1-2]
	\arrow[no head, from=2-3, to=1-2]
	\arrow[no head, from=3-2, to=2-1]
	\arrow[no head, from=3-2, to=2-3]
\end{tikzcd}
\vspace{10pt}
\captionof{figure}{The field diagram for Theorem \ref{finite of two q}}
\end{figure}

We choose a prime ideal $\mathfrak{q}_2'$ of $\mathcal{O}_{K(\zeta_p)}$ that is inert in $M(\mathfrak{q}_1,p)(\zeta_p)$ and is totally split over \(\mathbb{Q}\). Moreover, we assume that $\mathfrak{q}_2 = \mathfrak{q}_2' \cap \mathcal{O}_K$ is unramified in $K(\zeta_p)$ and that the prime number below $\mathfrak{q}_2$ does not divide the order of the group of roots of unity $\mu_K$. By the Chebotarev density theorem, there are infinitely many prime ideals of $\mathcal{O}_{K(\zeta_p)}$ that satisfy these conditions. By our choice of $\mathfrak{q}_2$, we have $N(\mathfrak{q}_2) \equiv 1 \pmod{p}$, and $\mathfrak{q}_2$ is inert in $M(\mathfrak{q}_1,p)$. Similarly, there is only one sub-extension $M(\mathfrak{q}_2,p)$ of the ray class field modulo $\mathfrak{q}_2$, and its Galois group over \(K\) is isomorphic to $\mathbb{Z}/p\mathbb{Z}$ by Corollary \ref{2.6}.  

Now, let \(S=\{\mathfrak{q}_1,\mathfrak{q}_2\}\) and $M$ be the fixed field by the Frattini subgroup of $G_{S}(K,p)$. Clearly, we have \(M(\mathfrak{q}_1,p) \cdot M(\mathfrak{q}_2,p) \subset M.\) On the other hand, by Proposition \ref{2.5}, we have
\begin{gather*}
 d(G_{S}(K,p))=1- r+\delta(K_{\mathfrak{q}_{1}})+\delta(K_{\mathfrak{q}_2})-\delta(K)+\dim_{\mathbb{F}_{p}}\Be_S(K)\\
 \le 1-r +\delta(K_{\mathfrak{q}_1})+\delta(K_{\mathfrak{q}_2})-\delta(K) \\
 +\dim_{\mathbb{F}_p}Cl_K(p
) + \dim_{\mathbb{F}_p}\mathcal{O}_K^{\times} /p
 \le 2.   
\end{gather*}
It follows that the generator rank of $G_{S}(K,p)$ equals 2. 
\begin{figure}[ht]
\begin{tikzcd}[scale cd=0.6]
	& {K_S} \\
	\\
	& M \\
	{M(\mathfrak{q}_2,p)} && {M(\mathfrak{q}_1,p)} \\
	& K
	\arrow[no head, from=3-2, to=1-2]
	\arrow[no head, from=4-1, to=3-2]
	\arrow[no head, from=4-3, to=3-2]
	\arrow[no head, from=5-2, to=4-1]
	\arrow[no head, from=5-2, to=4-3]
\end{tikzcd}
\vspace{10pt}
\captionof{figure}{The field diagram for Theorem \ref{finite of two q}}
\end{figure}

Moreover, by the choice of \(\mathfrak{q}_2\), the Galois group $\operatorname{Gal}(M /K)$ is isomorphic to the decomposition subgroup of $\mathfrak{q}_2$. The results follow from Lemma \ref{lem 3.5}. 

\end{proof}

\subsection{The case of even \texorpdfstring{$p$}{p}}

Furthermore, if we consider only the finiteness of \(G_S(K,p)\) when \(p=2\), we obtain the following results:

\begin{thm}\label{finite s=2 p=2}
Let $K$ be an imaginary quadratic field with trivial 2-class group. Let $\mathfrak{q}_1$ and $\mathfrak{q}_2$ be two distinct prime ideals of $K$ not lying above 2. Suppose $\mathfrak{q}_1$ satisfies $4 \mid N(\mathfrak{q}_1) - 1$. If $\mathfrak{q}_2$ is inert in $M(\mathfrak{q}_1,2)$ and $4 \mid N(\mathfrak{q}_2) - 1$, then the Galois group $G_S(K,2)$ is a finite 2-group, where $S = \{\mathfrak{q}_1, \mathfrak{q}_2\}$.
\end{thm}
\begin{proof}
As proof of the above theorem, the fixed field of the Frattini subgroup of $G_{S}(K,2)$, denoted by $M$, is the compositum of $M(\mathfrak{q}_1,2)$ and $M(\mathfrak{q}_2,2)$. The existence of $M(\mathfrak{q}_2,2)$ follows from Remark \ref{rk 2.7}. Thus, the Galois group $G_{S}(K,2)$ is equal to the decomposition subgroup of $G_{S}(K,2)$ at the prime $\mathfrak{Q}$ lying above $\mathfrak{q}_2$.

Let $I_{\mathfrak{Q}}$ denote the inertia subgroup of $\mathfrak{Q}$ and the fixed field $F$ is a cyclic extension of $K$ because the inertia subgroup of a tamely ramified extension is cyclic. Indeed, $F$ is a finite extension of $K$ since $D_{\mathfrak{Q}}/I_{\mathfrak{Q}}$ is a cyclic group, implying that $F$ is a subfield of $K_{S}^{\text{ab}}$. Furthermore, $K_{S}$ is a finite extension of $F$, as $\operatorname{Gal}(F_{S'}^{\operatorname{ab}}/K_S) \subseteq G_{S'}(F,2)^{\operatorname{ab}}$, where $S'$ consists of all primes lying above $\mathfrak{q}_1$ or $\mathfrak{q}_2$. Consequently, $K_{S}$ is a finite extension of $K$.
\end{proof}

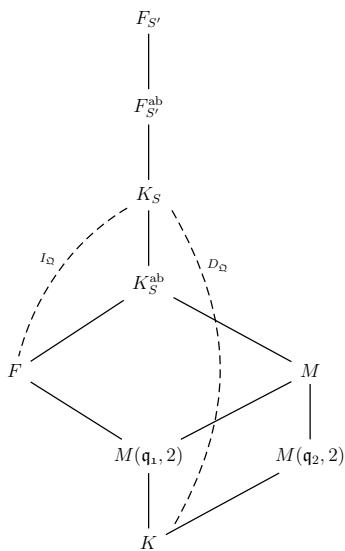
\begin{figure}[ht]
\begin{tikzcd}[scale cd=0.6]
	& {F_{S'}} \\
	& {F_{S'}^{\text{ab}}} \\
	& {K_S} \\
	& {K_S^{\text{ab}}} \\
	F && M \\
	& {M(\mathfrak{q_1},2)} & {M(\mathfrak{q}_2,2)} \\
	& K
	\arrow[no head, from=1-2, to=2-2]
	\arrow[no head, from=2-2, to=3-2]
	\arrow[no head, from=3-2, to=4-2]
	\arrow["{I_{\mathfrak{Q}}}"', curve={height=12pt}, dashed, no head, from=3-2, to=5-1]
	\arrow["{D_{\mathfrak{Q}}}"{pos=0.2}, shift left=2, curve={height=-24pt}, dashed, no head, from=3-2, to=7-2]
	\arrow[no head, from=5-1, to=4-2]
	\arrow[no head, from=5-3, to=4-2]
	\arrow[no head, from=5-3, to=6-3]
	\arrow[no head, from=6-2, to=5-1]
	\arrow[no head, from=6-2, to=5-3]
	\arrow[no head, from=7-2, to=6-2]
	\arrow[no head, from=7-2, to=6-3]
\end{tikzcd}
\vspace{10pt}
\captionof{figure}{The field diagram for Theorem 3.8.}
\end{figure}

\begin{thm}\label{finite s=1 p=2}
    
Let $K$ be an imaginary quadratic field with a non-trivial cyclic 2-class group. Let $\mathfrak{q}$ be a prime ideal of $K$ that is not split in the 2-class tower of $K$ and does not lie above 2. Then, $G_{S}(K,2)$ is finite where $S=\{\mathfrak{q}\}$.
\end{thm}

\begin{proof}
By the Proposition \ref{2.5}, it is straightforward to determine that the generator rank of $G_{S}(K,2)$ is at most 2. Without loss of generality, assume $d(G_{S}(K,2)) = 2$. In this scenario, $4 \mid (N(\mathfrak{q}) - 1)$, and there exists exactly one $\mathbb{Z}/2\mathbb{Z}$-extension $M(\mathfrak{q},2)$ of $K$, which is totally ramified in $\mathfrak{q}$.

Let $M$ be the field fixed by the Frattini subgroup of $G_{S}(K,2)$, and let $H_2(K)$ be the unique subfield of the 2-Hilbert class field of $K$ contained in $M$. Then $\mathfrak{q}$ does not split in $M$. By Burnside's basis theorem, $G_{S}(K,2)$ is the decomposition subgroup of $\mathfrak{Q}$, a prime of $K_{S}$ lying over $\mathfrak{q}$. Therefore, $K_{S}$ is a finite extension of $K$.
\end{proof}

\section{\texorpdfstring{The structure of \(G_S(K,p)\)}{The structure of G\_S(K,p)}}

For the structure of the Galois group \(G_S(K,p)\), even if we know that $G_{S}(K,p)$ is finite in certain cases, its structure is often difficult to determine. For a general number field $K$, it is not always clear whether $K_{S}$ is an abelian extension of $K$. However, we have the following two helpful facts:

Let \(G\) be a profinite group. The second homology group \(H_2(G, \mathbb{Z})\), with \(\mathbb{Z}\) acting trivially, is called the Schur multiplier of \(G\), denoted \(M(G)\). The following theorem, established by Schur in 1904, is a well-known result about the Schur multiplier.

\begin{thm}[\cite{karpilovsky1987schur}, Proposition 2.1.1 and Proposition 2.5.14]
If \(G\) is a cyclic finite group, then \(M(G) = 1\). Furthermore, if \[G \cong \mathbb{Z}/n_1\mathbb{Z} \oplus \mathbb{Z}/n_2\mathbb{Z} \oplus \ldots \oplus \mathbb{Z}/n_k\mathbb{Z},\] where \(n_{i+1} \mid n_i\) for all \(i \in \{1, 2, \ldots, k-1\}\) and \(k \geq 2\), then 
$$
M(G) \cong \mathbb{Z}/n_2\mathbb{Z} \oplus (\mathbb{Z}/n_3\mathbb{Z})^2 \oplus \ldots \oplus (\mathbb{Z}/n_k\mathbb{Z})^{k-1}.
$$
\end{thm}

\begin{cor}\label{non-schur}
    The non-cyclic finite Abelian group has a non-trivial Schur multiplier.
\end{cor}

The above corollary indicates that if $K_S$ is a finite extension of $K$, a good method to determine whether the group $G_S(K,p)$ is Abelian is to calculate the Schur multiplier. In fact, in 1985, Watt provided the following key theorem:

\begin{thm}\cite{watt1985restricted}\label{thm 4.3}
    Let \(K\) be an imaginary quadratic field that does not contain \(p\)-th roots of unity. Then, the pro-\(p\) group \(G_S(K,p)\) is multiplier-free for any finite set \(S\) of finite primes of \(K\).
\end{thm}

If \(G_S(K,p)\) is finite, the definition of a multiplier-free group being equivalent to the Schur multiplier is trivial. For further details, the reader may refer to \cite{watt1985restricted}, Section 1, and \cite{frohlich1983central}, Equation 3.3. 

At the same time, the following lemma provides essential information regarding our results.

\begin{lem}[\cite{gorenstein2007finite}, Chapter 5, Theorem 4.4 adn Theorem 4.3]\label{extra group}
Let $p$ be an odd prime and $P$ a nonabelian $p$-group of order $p^n$ that contains a cyclic subgroup of order $p^{n-1}$. Then, $P$ is an extraspecial group with presentation $$\left\langle a,b \bigl\vert a^{p^{n-1}} = 1, \, b^p = 1, \, b^{-1}ab = a^{1+p^{n-2}} \right\rangle.$$
\end{lem}

Now, set $p$ be an odd rational prime number. We provide the proof of our main theorems.

\subsection{\texorpdfstring{The case of \( S=\{\mathfrak{q}\} \)}{The case of S={q}}}
When \( S \) consists of exactly one prime, the \( p \)-divisibility of the ray class number does not provide sufficient information to control the generator rank of \( G_S(K) \). Additional information about the ray class group is necessary.

We recall a result of Gras-Munnier that provides a criterion for the existence of a totally ramified \(\mathbb{Z}/p\mathbb{Z}\)-extension at a given set \(S\). Let \(K\) be a number field. The governing field of \(K\) is the field \(K(\zeta_p, \sqrt[p]{V_{\varnothing}})\), obtained by adjoining the primitive \(p\)-th roots of unity and the \(p\)-th roots of elements in \(V_{\varnothing}\) to \(K\). We denote this field by \(\operatorname{Gov}(K)\).

\begin{thm}
    [Gras-Munnier,\cite{gras1998extensions}] Let \(S=\{\mathfrak{q}_1,\cdots, \mathfrak{q}_n\}\) be a finite set of places of \(K\) coprime to \(p\). There exists a cyclic degree \(p\) unramified extension \(L/K\) outside \(S\) and totally ramified at each place of \(S\), if and only if, for \(i=1,...,n\), there exists \(c_i\in (\mathbb{Z}/p\mathbb{Z})^{\times}\), such that \[\prod_{\mathfrak{q}_i\in S}\left(\frac{\operatorname{Gov}(K)/K(\zeta_p)}{\mathfrak{q}_i}\right)^{c_i}=1\in \operatorname{Gal}(\operatorname{Gov}(K)/K(\zeta_p)).\]
\end{thm}

The Gras-Munnier theorem is a powerful tool for calculating the ray class group and the Shafarevich group of restricted ramification. For further details, refer to \cite{gras1998extensions} and \cite{hajir2019shafarevich}. 
\\

\textbf{The proof of Theorem \ref{s=1}.}
By the Chebotarev density theorem, there are infinitely many primes $\mathfrak{q}' \subseteq \operatorname{Gov}(K)$ that are inert in $\text{Gov}(K) \cdot H_p(K)$ with absolute degree 1. Furthermore, we can select $\mathfrak{q}'$ such that it is totally split over $\mathbb{Q}$. Let $\mathfrak{q} = \mathfrak{q}' \cap \mathcal{O}_K$. It is straightforward to verify that the prime $\mathfrak{q}$ is inert in $H_p(K)$, and all the Frobenius automorphisms of $\operatorname{Gal}(\operatorname{Gov}(K)/K(\zeta_p))$ at primes lying above $\mathfrak{q}$ are trivial. 

By the Gras-Munnier theorem, there exists a cyclic degree $p$ unramified extension of $K$ outside $\mathfrak{q}$, denoted by $M(\mathfrak{q},p)$. Moreover, as shown in the proof of Lemma \ref{lem2.4}, $M(\mathfrak{q},p)$, as a $p$-extension of $K$, is a subfield of $K(\mathfrak{q})$. If the $p$-rank of $Cl_K(\mathfrak{q})$ equals 1, then $M(\mathfrak{q},p)$ is the unique cyclic degree $p$ subfield of $K(\mathfrak{q})$. This leads to a contradiction since $K$ has a non-trivial $p$-class group. Thus, we ascertain the first results.


\begin{figure}[!ht]
\begin{tikzcd}[scale cd=0.7]
	& {H_p(K)(\mathfrak{q}')} \\
	& {K_S} \\
	\\
	{\operatorname{Gov}(K) \cdot H_p(K)} && M \\
	{\operatorname{Gov}(K)} & {H_p(K)} \\
	{K(\zeta_p)} && {M(\mathfrak{q},p)} \\
	& K
	\arrow[no head, from=2-2, to=1-2]
	\arrow["{I_{\mathfrak{Q}}}"', curve={height=12pt}, dashed, no head, from=2-2, to=5-2]
	\arrow[no head, from=4-1, to=5-1]
	\arrow[no head, from=4-3, to=2-2]
	\arrow[no head, from=5-2, to=4-1]
	\arrow[no head, from=5-2, to=4-3]
	\arrow["{\mathbb{Z}/p \mathbb{Z}}"', curve={height=12pt}, dashed, no head, from=5-2, to=7-2]
	\arrow[no head, from=6-1, to=5-1]
	\arrow[no head, from=6-3, to=4-3]
	\arrow[no head, from=7-2, to=5-2]
	\arrow[no head, from=7-2, to=6-1]
	\arrow[no head, from=7-2, to=6-3]
\end{tikzcd}
\vspace{10pt}
\captionof{figure}{The field diagram for Theorem \ref{s=1}}
\end{figure}
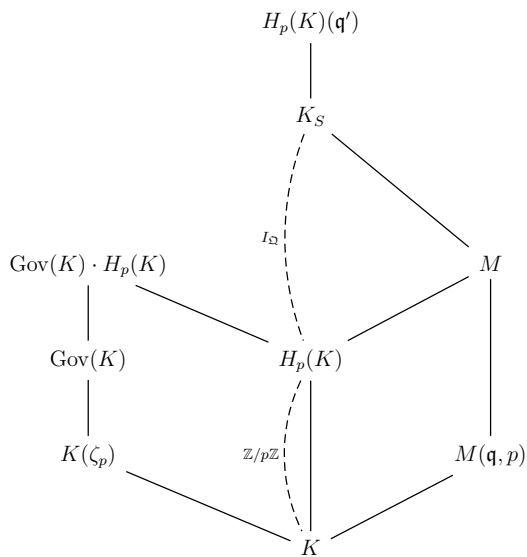

Now, we fix a prime ideal $\mathfrak{q}$ that satisfies the above properties. Let $S = \{\mathfrak{q}\}$, and let $M$ denote the field fixed by the Frattini subgroup of $G_{S}(K,p)$. Then, we have 
\(
H_{p}(K) \cdot M(\mathfrak{q}, p) \subseteq M.
\)
On the other hand, since \(p\vert N(\mathfrak{q})-1\), we have
\begin{align*}
d(G_{S}(K,p)) & \le 1 - r + \delta(K_{\mathfrak{q}}) - \delta(K) \\
& \quad + \dim_{\mathbb{F}_{p}} Cl_{K}(p) + \dim_{\mathbb{F}_{p}} \mathcal{O}_{K}^{\times}/p \le 2.
\end{align*}
It follows that the generator rank of \(G_{S}(K,p)\) is 2. In other words, \(G_{S}(K,p)\) is not a procyclic group.

By Lemma \ref{lem 3.5} and Theorem \ref{thm 4.3}, the Galois group $G_{S}(K,p)$ is a powerful finite $p$-group with a trivial Schur multiplier. On the other hand, it follows from Corollary \ref{non-schur} that $G_{S}(K,p)$ is non-Abelian. 

Let $\mathfrak{Q}$ be a prime of $K_{S}$ lying above $\mathfrak{q}$. Then the decomposition subgroup of $G_{S}(K,p)$ at $\mathfrak{Q}$, denoted by $D _{\mathfrak{Q}}$, equals $G_{S}(K,p)$. Let $I _{\mathfrak{Q}}$ be the inertia subgroup of $D _{\mathfrak{Q}}$. By class field theory, $I _{\mathfrak{Q}}$ is a cyclic $p$-group whose fixed field is the $p$-Hilbert class field of $K$. 
Since the $p$-class group of $K$ is isomorphic to $\mathbb{Z} /p \mathbb{Z}$, if the order of $G_{S}(K,p)$ equals $p ^{n}$, the order of $I _{\mathfrak{Q}}$ equals $p ^{n-1}$. By Lemma \ref{extra group}, $G_{S}(K,p)$ is an extraspecial group. In other words, $$ G_{S}(K,p) = \left\langle a, b \mid a^{p^{n-1}} = 1, b^{p} = 1, b^{-1}ab = a^{1+p^{n-2}} \right\rangle.$$  

Moreover, $K_{S}$ is a cyclic unramified extension of $H_{p}(K)$ outside $\mathfrak{q}'$, which is the prime ideal of $H_{p}(K)$ lying above $\mathfrak{q}$. The maximal \(p\)-extension of \(H_p(K)\) contained in \(H_p(K)(\mathfrak{q}')\) is a Galois extension of \(K\) as \(\mathfrak{q}\) is inert in \(H_p(K)\). By Lemma \ref{lem2.4}, $p^{n-1} \Vert \left|Cl_{H_p(K)}(\mathfrak{q}')\right|$. We have   
$[K_{S}:K] = p \cdot |(Cl _{H _{p}(K)}(\mathfrak{q}'))(p)|.$
\qed

\begin{cor}\label{p^3}
    If the non-$p$-adic prime ideal $\mathfrak{q}$ has the properties and additionally satisfies $p \Vert N(\mathfrak{q})-1$, then the order of $G_S(K,p)$ equals $p^3$.
\end{cor}
\begin{proof}
By \cite{lemmermeyer2010class}, Proposition 1.9.2, the $p$-class group of $H_p(K)$ is trivial. Since $\mathfrak{q}$ is inert in $H_p(K)$, the ideal norm of $\mathfrak{q}'$, the prime ideal of $H_p(K)$ lying above $\mathfrak{q}$, equals $N(\mathfrak{q})^p$. It follows that the order of $I_{\mathfrak{Q}}$ is at most $p^2$.

If $I_{\mathfrak{Q}}$ is isomorphic to $\mathbb{Z}/p\mathbb{Z}$, then $K_S$ would equal the field $M$, which is an abelian extension of $K$. This is a contradiction by Corollary \ref{non-schur}. Hence, $G_S(K,p)$ is an extraspecial group of order $p^3$.
\end{proof}


\subsection{The proof of the Theorem \ref{thm 1.1}}

As the proof of Theorem \ref{finite of two q} showed, $G_S(K,p)$ is a finite powerful group with two generators, and it equals the decomposition subgroup of $G_S(K,p)$ at $\mathfrak{Q}_2$, a prime of $K_S$ lying above $\mathfrak{q}_2$. Let $I_{\mathfrak{Q}_2}$ be the inertia subgroup at $\mathfrak{Q}_2$, with its fixed field denoted by $T_{\mathfrak{q}_2}$.

By class field theory, the inertia field $T_{\mathfrak{q}_2}$ is a cyclic extension over $K$, and we have $M(\mathfrak{q}_1,p) \subseteq T_{\mathfrak{q}_2} \subseteq K(\mathfrak{q}_1^l)$ for some $l$. In fact, by Lemma \ref{lem2.4}, $l=1$. Moreover, $M(\mathfrak{q}_1,p) = T_{\mathfrak{q}_2}$, by Lemma \ref{lem 2.5}, and $p \mid\mid N(\mathfrak{q}_1) - 1$. 

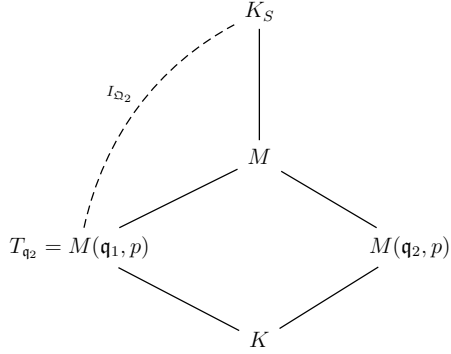
\begin{figure}[ht]
\begin{tikzcd}[scale cd=0.7]
	& {K_S} \\
	\\
	& M \\
	{T_{\mathfrak{q}_2}=M(\mathfrak{q}_1,p)} && {M(\mathfrak{q}_2,p)} \\
	& K
	\arrow[no head, from=3-2, to=1-2]
	\arrow["{I_{\mathfrak{Q}_2}}", curve={height=-18pt}, dashed, no head, from=4-1, to=1-2]
	\arrow[no head, from=4-1, to=3-2]
	\arrow[no head, from=4-3, to=3-2]
	\arrow[no head, from=5-2, to=4-1]
	\arrow[no head, from=5-2, to=4-3]
\end{tikzcd}
\vspace{10pt}
\captionof{figure}{The field diagram for Theorem \ref{thm 1.1}}
\end{figure}
It follows that $G_S(K,p)$ has a maximal $p$-subgroup $I_{\mathfrak{Q}_2}$. On the other hand, $G_S(K,p)$ is non-abelian and non-procyclic, since $G_S(K,p)$ has generator rank 2 and a trivial Schur multiplier. By Lemma \ref{extra group}, $$G_S(K) = \left\langle a,b \mid a^{p^{n-1}} = 1, b^p = 1, b^{-1}ab = a^{1+p^{n-2}}\right \rangle,$$ where $p^n = \text{ord}(G_S(K)).$

Similar to the discussion in the proof of Theorem \ref{s=1}, $K_S$ is an abelian extension over $M(\mathfrak{q}_1,p)$, contained in $M(\mathfrak{q}_1,p)(\mathcal{Q}_1^{n_1}\mathcal{Q}_2^{n_2})$ for some $n_1, n_2 \in \mathbb{N}^{\times}$. However, by Lemma \ref{lem2.4}, $p^m \mid\mid \left|Cl_{M(\mathfrak{q}_1,p)}(\mathcal{Q}_1^{n_1}\mathcal{Q}_2^{n_2})\right|$ if and only if $p^m \mid\mid \left|Cl_{M(\mathfrak{q}_1,p)}(\mathcal{Q}_1 \mathcal{Q}_2)\right|$. Here, $\mathcal{Q}_i$ are the prime ideals of $M(\mathfrak{q}_1, p)$ lying above $\mathfrak{q}_i$ for $i = 1, 2$. 
Since $\mathfrak{q}_1$ is totally ramified and $\mathfrak{q}_2$ is inert in $M(\mathfrak{q}_1, p)$, the maximal $p$-extension of $M(\mathfrak{q}_1, p)$ contained in $M(\mathfrak{q}_1, p)(\mathcal{Q}_1 \mathcal{Q}_2)$ is a Galois extension of $K$.
Hence, the number $n-1$ is the largest such that $p^{n-1}$ divides the ray class number of $M(\mathfrak{q}_1,p)$ modulo $\mathcal{Q}_1 \mathcal{Q}_2$.

\qed

\section{Application}

\begin{ex}
    Let \(K=\mathbb{Q}(\sqrt{-23})\), whose class group is isomorphic to \(\mathbb{Z}/3 \mathbb{Z}\) by \cite{lmfdb}. The prime \(151\) factors into the prime ideals \(\mathfrak{q}_1=(151, \frac{\sqrt{-23} + 85}{2})\) and \(\mathfrak{q}_2=(151, \frac{\sqrt{-23}+217}{2})\) in \(K\). The Hilbert class field \(H(K)\) of \(K\) is defined by the polynomial 
\[
x^6 - 2x^5 + 70x^4 - 90x^3 + 1631x^2 - 1196x + 12743.
\]

By \cite{sagemath}, the ideals \(\mathfrak{q}_i\) for \(i=1,2\) are inert in \(H(K)\), and the \(p\)-rank of \(Cl_{H(K)}(\mathfrak{q}_i)\) is 2. Let \(S\) be either \(\{\mathfrak{q}_1\}\) or \(\{\mathfrak{q}_2\}\). Then \(G_S(K,3)\) is a finite group of order 27 by Corollary \ref{p^3}.
\end{ex}

\begin{ex}\label{ex1}
    Let $K =\mathbb{Q}(i)$ and $\mathfrak{q}_1=7\mathcal{O}_K$. By ramification theory, the field $M(\mathfrak{q}_1,3)=K(\zeta _{7}+ \zeta_7 ^{-1})$. Since 31 is inert in $\mathbb{Q}(\zeta _{7})$, the prime ideal $31 \mathcal{O} _{K}$ is inert in $M(\mathfrak{q}_1,3)$. Now, set $\mathfrak{q}_2=31\mathcal{O} _{K}$ and $S=\{\mathfrak{q}_1,\mathfrak{q}_2\}$. By \cite{sagemath}, we have $3^2 \Vert  \left|Cl_{M (\mathfrak{q}_1,3)}(\mathfrak{q}_1' \mathfrak{q}_2')\right|$. It follows that 
$$ G_{S}(K,3) = \left\langle a,b\mid a^{3^{2}}=1, b^{3}=1, b^{-1}ab=a^{4} \right\rangle,$$ and $|G_{S}(K,3)|=27$  by the Theorem \ref{thm 1.1}. We determine all intermediate fields of $K_S$, which are illustrated in the following diagram and table (computed using \cite{sagemath} and \cite{lmfdb}).
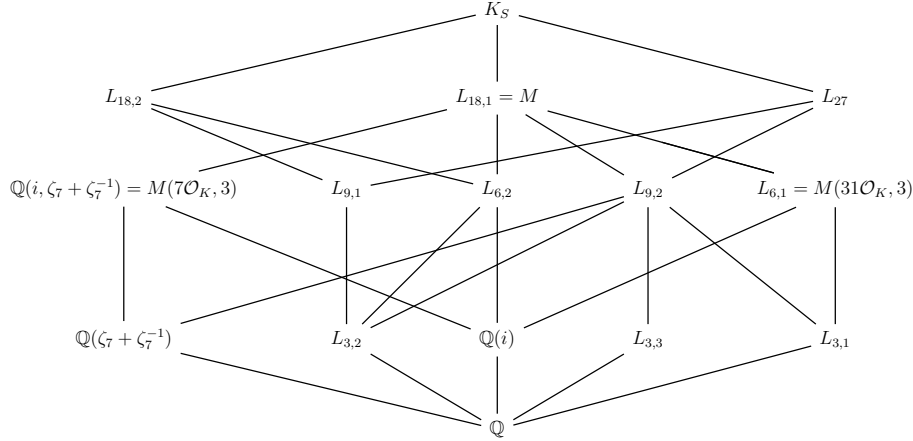
\begin{figure}[ht]
\begin{tikzcd}[scale cd=0.6]
	&& {K_S} \\
	{L_{18,2}} && {L_{18,1}=M} && {L_{27}} \\
	{\mathbb{Q}(i,\zeta_7+\zeta_7^{-1})=M(7\mathcal{O}_K,3)} & {L_{9,1}} & {L_{6,2}} & {L_{9,2}} & {L_{6,1}=M(31\mathcal{O}_K,3)} \\
	\\
	{\mathbb{Q}(\zeta_7+\zeta_7^{-1})} & {L_{3,2}} & {\mathbb{Q}(i)} & {L_{3,3}} & {L_{3,1}} \\
	&& {\mathbb{Q}}
	\arrow[no head, from=2-1, to=1-3]
	\arrow[no head, from=2-3, to=1-3]
	\arrow[no head, from=2-5, to=1-3]
	\arrow[no head, from=3-1, to=2-3]
	\arrow[no head, from=3-2, to=2-1]
	\arrow[no head, from=3-2, to=2-5]
	\arrow[no head, from=3-3, to=2-1]
	\arrow[no head, from=3-3, to=2-3]
	\arrow[no head, from=3-4, to=2-3]
	\arrow[no head, from=3-4, to=2-5]
	\arrow[no head, from=3-4, to=5-1]
	\arrow[no head, from=3-4, to=5-2]
	\arrow[no head, from=3-4, to=5-5]
	\arrow[""{name=0, anchor=center, inner sep=0}, no head, from=3-5, to=2-3]
	\arrow[no head, from=5-1, to=3-1]
	\arrow[no head, from=5-2, to=3-2]
	\arrow[no head, from=5-2, to=3-3]
	\arrow[no head, from=5-3, to=3-1]
	\arrow[no head, from=5-3, to=3-3]
	\arrow[no head, from=5-3, to=3-5]
	\arrow[no head, from=5-3, to=6-3]
	\arrow[no head, from=5-4, to=3-4]
	\arrow[no head, from=5-5, to=3-5]
	\arrow[no head, from=6-3, to=5-1]
	\arrow[no head, from=6-3, to=5-2]
	\arrow[no head, from=6-3, to=5-4]
	\arrow[no head, from=6-3, to=5-5]
	\arrow[no head, from=0, to=3-5]
\end{tikzcd}
\vspace{10pt}
\captionof{figure}{The field diagram for Example 5.2.}
\end{figure}

\begin{table}[ht]
\centering
\caption{Number Fields and LMFDB Labels}
\vspace{5pt}
\begin{tabular}{cccc}
\toprule
Number field & Field label   \\
\midrule
$L_{3,1}$ & $3.3.961.1$  \\
\hline
$L_{3,2}$ &$3.3.47089.1$  \\
\hline
$L_{3,3}$ & $3.3.47089.2$ \\
\hline
$L_{6,1}$ & $6.0.59105344.1$ \\
\hline
$L_{6,2}$ & $6.0.141911930944.3$ \\
\hline
$L_{9,1}$ & $9.9.4916747105530914241.1$ \\
\hline
$L_{9,2}$ & $9.9.104413920565969.1$ \\
\hline
$L_{18,1}$ & $18.0.2857963830104944567197606598672384.1$ \\
\bottomrule
\end{tabular}
\end{table}

In this case, the polynomial defined for the number field $K_S$ is \(P_{7,31}(x)\) as described in the Appendix.
\end{ex}

\bibliographystyle{plain}
\bibliography{refs}

\section*{Appendix}

\begin{tiny}

\begin{align*}
&P_{7,31}(x)=x^{54} + 12x^{53} - 1783x^{52} - 24920x^{51}\\
& + 1418742x^{50} + 23124830x^{49} - 658094536x^{48}\\
& - 12754104138x^{47} + 193582250939x^{46}\\
& + 4680549437534x^{45} - 35897937247714x^{44}\\
& - 1210424844289994x^{43} + 3459824891669503x^{42}\\
& + 227523102612751604x^{41} + 127825646419374612x^{40}\\
& - 31541031306199512502x^{39} - 103299433885895897754x^{38}\\
& + 3227173503047895228474x^{37} + 18577455849803742456606x^{36}\\
& - 239759386799832388807378x^{35} - 2048499696439510191458935x^{34}\\
& + 12298544414109879128683730x^{33} + 157509170731938015861566988x^{32}\\
& - 366889537016519263395405070x^{31} - 8781237213967246977251985414x^{30}\\
& - 138324655529723611214473178x^{29} + 359226295725688104559332528273x^{28}\\
& + 625759658374561519272412906114x^{27} - 10765185515516421149809454284848x^{26}\\
& - 34862502706534531399521875987902x^{25} + 232552988078673642556581149534141x^{24}\\
& + 1132003296610349707197273072759910x^{23} - 3446648750438089011731448922435768x^{22}\\
& - 25263466938106499293905633321155672x^{21} + 29384527492948148265182285831365346x^{20}\\
& + 405065325676044462096496754541512902x^{19} + 12379120881086439400998766806701310x^{18}\\
& - 4690834147962044903702559959261211820x^{17} - 4266983107150271135557860518961187755x^{16}\\
& + 38316273883851502485452647821640584338x^{15} + 64162753691071892616283994573735162388x^{14}\\
& - 205758780841891152376695682905367453016x^{13} - 524079475332910238054658555904263273459x^{12}\\
& + 580982406052016036018682093828948696766x^{11} + 2521172565067605447472641176250136066026x^{10}\\
& + 213668492673402676736308936810588517738x^{9} - 6115834439821560312598482128958360207937x^{8}\\
& - 6508718347371438155981223455611058165130x^{7} + 1687882960398434867905795379450225200583x^{6}\\
& + 10261569509295854003735817154407130798188x^{5} + 12310847145272292985677066754811692755603x^{4}\\
& + 17973891694655829390710173836105859704330x^{3} + 37277823726992262812107481960369177609036x^{2}\\
& + 35866920771719365499542605146210629657136x + 24964752719863841282374259624636967389453.
\end{align*}
\end{tiny}

\end{document}